\documentclass[leqno,11pt]{article}
\usepackage{amssymb,amsmath,amsthm,amsfonts}
\usepackage{verbatim}
\allowdisplaybreaks

\usepackage[utf8]{inputenc}

\usepackage{enumitem}

\theoremstyle{plain}
\newtheorem{theorem}{Theorem}
\newtheorem{lemma}[theorem]{Lemma}
\newtheorem{corollary}[theorem]{Corollary}
\newtheorem{proposition}[theorem]{Proposition}

\theoremstyle{definition}
\newtheorem{definition}[theorem]{Definition}

\newtheorem{question}[theorem]{Question}
\newtheorem{conjecture}[theorem]{Conjecture}
\newtheorem{notation}[theorem]{Notation}

\newcommand{\R}{\mathbb R}

\newcommand{\sS}{\mathbb S}
\newcommand{\B}{\mathbb B}
\newcommand{\cL}{P}

\numberwithin{theorem}{section} \numberwithin{equation}{section}

\title{BiLipschitz embeddings of spheres into jet space Carnot groups not admitting Lipschitz extensions}
\author{Derek Jung\footnote{Supported  by U.S. Department of Education GAANN fellowship P200A150319. \hfill\break {\it Key Words and Phrases:} sub-Riemannian geometry, biLipschitz embeddings, jet spaces, Carnot groups, Lipschitz extensions \hfill\break {\it 2010 Mathematics Subject Classification:} Primary 53C17, 58A20; Secondary 30L05, 26A16, 22E25} \\
University of Illinois at Urbana-Champaign \\ 
}
\date{\today}

\begin{document}

\raggedbottom

\maketitle

\begin{abstract}
For all $k,n\ge 1$, we construct a biLipschitz embedding of $\mathbb{S}^n$ into the jet space Carnot group $J^k(\mathbb{R}^n)$ that does not admit a Lipschitz extension to $\mathbb{B}^{n+1}$. 
Let $f:\mathbb{B}^n\to \mathbb{R}$ be a smooth, positive function with $k^{th}$-order derivatives that are approximately linear near $\partial \mathbb{B}^n$.
The embedding is given by taking the jet of $f$ on the upper hemisphere and the jet of $-f$ on the lower hemisphere, where we view $\mathbb{S}^n$ as two copies of $\mathbb{B}^n$. 
To prove the lack of a Lipschitz extension, we apply a factorization result of Wenger and Young for $n=1$ and modify an argument of Rigot and Wenger for $n\ge 2$.
\end{abstract}

%%%%%%%%%%%%%%%%%%%%%%%%%%%%%%%%%%%%%%%%%%%%%%%%%%%
%%%%%%%%%%%%%%%%%%%%%%%%%%%%%%%%%%%%%%%%%%%%%%%%%%%
%%%%%%%%%%%%%%%%%%%%%%%%%%%%%%%%%%%%%%%%%%%%%%%%%%%

\tableofcontents
 
\setcounter{section}{0}
\setcounter{subsection}{0}

%%%%%%%%%%%%%%%%%%%%%%%%%%%%%%%%%%%%%%%%%%%%%%%%%%%
%%%%%%%%%%%%%%%%%%%%%%%%%%%%%%%%%%%%%%%%%%%%%%%%%%%
%%%%%%%%%%%%%%%%%%%%%%%%%%%%%%%%%%%%%%%%%%%%%%%%%%%

\section{Introduction}

The existence of  extensions that preserve regularity  is a topic that   permeates  mathematics, especially in topology and analysis. 
In topology, one has the famous Tietze Extension Theorem.
In differential geometry, while one cannot smoothly extend \emph{any} smooth function defined on a subset of a manifold, one may if the subset is assumed to be closed (see for instance  \cite[Lemma 2.27]{L:ITS}).
An essential result of functional analysis  in the same vein is the Hahn-Banach Theorem from functional analysis.
These three results all confirm the existence of  extensions  that preserve the ``right'' regularity based on the context. 
Indeed, one can preserve \emph{continuity} for normal topological spaces, \emph{smoothness} for manifolds, and \emph{boundedness} for Banach spaces. 
For Carnot groups, the lack of a linear structure combined with Rademacher's Theorem and Pansu's generalization suggest that \emph{Lipschitz} is the ``right'' form of regularity to consider. 
In this paper, we will be interested in Lipschitz extensions of mappings into Carnot groups. 

%Extend partially-defined maps rich area research. 
%The reader may have interpreted the previous expression to mean: The question of whether one can suitably extend partially-defined maps forms a rich area of research. 
%Or perhaps  the reader wasn’t familiar with this paper’s content and  thought of: The earliest explorers sought to extend partially-defined maps in search of rich areas and research. 
%And a third option would be: A random list of words is extend, partially-defined, maps, rich, area, research. 
%Yet, this last choice is the least satisfying because it fails to maintain the structure inherent in the initial collection of words. 
%In this paper, we will study an analogous problem in math, whether one can fill in the gaps of a partially-defined mapping while preserving its  regularity. 

%In this paper, we will be interested in Lipschitz extensions. 
The well-known  McShane-Whitney Extension Theorem states that  every Lipschitz function defined on a subset of  a metric space can be wholly  extended  in a Lipschitz fashion, while preserving the Lipschitz constant (see for instance \cite[Chapter 4]{HKST:SS}).
%In fact, for any metric space $X$ and subset $A\subset X$, every $L$-Lipschitz function $f:A\to \R$ can be extended to an $L$-Lipschitz function $F:X\to \R$.
If one allows for a larger Lipschitz constant, one can replace $\R$ with $\R^n$.
With the finite-dimensional vector space case well-understood, other metric spaces have been considered as targets.  
Lipschitz extension results have been shown for mappings into Banach spaces and spaces of bounded curvature (see for instance \cite{JL:E, JLS:E, LN:EL} and  \cite{LS:ND, LS:K, V:C}, respectively).
Over the past decade, the problem for Carnot groups  has  drawn considerable attention, primarily for the Heisenberg groups and, more generally, jet space Carnot groups  \cite{DHLT:OTL, HST:HG, RW:LNE, WY:LE, WY:LHG}.
In this paper, we will be interested in considering the problem for the latter class.

In 2010, Rigot and Wenger proved that there exists a Lipschitz mapping from $\sS^n$ to $J^k(\R^n)$ that cannot be extended in a Lipschitz way to $\B^{n+1}$   \cite[Theorem 1.2]{RW:LNE}. 
%provided an example of a Lipschitz mapping $f :\partial [0,1]^{n+1} \to J^k(\R^n)$ that does not admit a Lipschitz extension $F:[0,1]^{n+1}\to J^k(\R^n)$ \cite[Theorem 1.2]{RW:LNE}. 
For their proof, they actually construct  a Lipschitz mapping $f :\partial [0,1]^{n+1} \to J^k(\R^n)$ that does not admit a Lipschitz extension to $[0,1]^{n+1}$. 
%This implies that there exists a Lipschitz mapping $g:\sS^n \to J^k(\R^n)$ that cannot extended to a Lipschitz mapping on $\mathbb{B}^{n+1}$. 
Their mapping $f$ is constant on each line $\{x\}\times [0,1]$, $x\in \partial [0,1]^{n}$, and, in particular, is not biLipschitz.
In this paper, we provide an explicit construction of  a biLipschitz embedding of $\sS^n$ into $J^k(\R^n)$ that cannot be Lipschitz extended to $\B^{n+1}$. 

\begin{theorem}\label{main-theorem}
	For all $k,n\ge 1$, there exists a biLipschitz embedding $\phi:\sS^n \to J^k(\R^n)$ that does not admit a Lipschitz extension $\tilde{\phi}:\B^{n+1}\to J^k(\R^n)$. 
\end{theorem}

We remark that the theorem's statement would be false if we replaced $\sS^n$ with a lower dimensional sphere. 
Wenger and Young proved that every biLipschitz embedding of $\sS^m$ into $J^k(\R^n)$, $m<n$,  can be extended to $\B^{m+1}$ in a Lipschitz fashion \cite[Theorem 1.1]{WY:LE}.

BiLipschitz embeddings of spheres into Carnot groups have been used to prove the nondensity of  Lipschitz mappings in Sobolev spaces. 
In 2009, Balogh and F\"assler provided an example of a horizontal embedding $\phi:\sS^n\to \mathbb{H}^n$  that does not admit a Lipschitz extension $\tilde{\phi}:\mathbb{B}^{n+1}\to \mathbb{H}^n$ \cite[Theorem 1]{BF:RAL}. 
Their example consisted of the Legendrian lift of a Lagrangian map $f:\sS^n\to \R^{2n}$.
Dejarnette, Haj{\l}asz, Lukyanenko, and Tyson then proved in 2014 that every horizontal embedding $\phi:\sS^n\to \mathbb{H}^n$ does not admit a Lipschitz extension to $\mathbb{B}^{n+1}$ \cite[Proposition 4.7]{DHLT:OTL}.
The last authors  used such an embedding to prove that  the collection of Lipschitz mappings $\text{Lip}(\mathbb{B}^{n+1},\mathbb{H}^n) $ is not dense in the Sobolev space $W^{1,p}(\B^{n+1}, \mathbb{H}^n)$ for $n\le p <n+1$ \cite[Proposition 1.3]{DHLT:OTL}. 
Haj{\l}asz, Schikorra, and Tyson have also horizontal embedding to prove the non-density of Lipschitz mappings in Heisenberg group-valued Sobolev spaces  \cite[Theorem 1.9]{HST:HG}.
%Also in 2014, Haj{\l}asz, Schikorra, and Tyson  used this biLipschitz embedding and the Hopf map to prove that $\text{Lip}(\B^{4n}, \mathbb{H}^{2n})$ is not dense in $W^{1,p}(\B^{4n},\mathbb{H}^{2n})$ for $4n-1 \le p <4n$ \cite[Theorem 1.9]{HST:HG}. 
Theorem \ref{main-theorem} is a step towards proving the following non-approximation result for  $J^k(\R^n)$:

\begin{conjecture}\label{conjecture-nondensity}
Lipschitz mappings $\text{Lip}(\B^{n+1}, J^k(\R^n))$ are not dense in $W^{1,p}(\B^{n+1}, J^k(\R^n))$, when $n\le p <n+1$. 
\end{conjecture}

All  smooth horizontal embeddings of $\sS^n$ into $\mathbb{H}^n$ are biLipschitz \cite[Theorem 3.1]{DHLT:OTL}. 
%Smooth horizontal embeddings of $\sS^n$ into $\mathbb{H}^n$ are easily seen to be biLipschitz due to the fact that their inverses are smooth (see Section 3 of \cite{DHLT:OTL} for examples of such embeddings). 
The difficulty of proving that our embedding $\phi:\mathbb{S}^n\to J^k(\R^n)$ is biLipschitz will stem from the fact that it is  not smooth along the equator of $\sS^n$. 
In fact, $\phi$ will not even be differentiable at these points. 
Fortunately, $\phi$ will be horizontal when restricted to the lower and upper  hemispheres, which will imply that our embedding is biLipschitz when restricted to either of these halves.
Still, the lack of differentiability begs the following question:

\begin{question}
For $n\ge 2$, does there exist a smooth, horizontal embedding $\psi:\sS^n\hookrightarrow J^k(\R^n)$ that does not admit a Lipschitz extension to $\B^{n+1}$?
\end{question}

In Section \ref{background-section}, we review the structure of jet space Carnot groups and state notation. 
In Section \ref{circle-section}, we prove Theorem \ref{main-theorem} for $n=1$ and observe that $\pi_m^{Lip}(J^k(\R)) =0$ for all $m\ge 2$ and $k\ge 1$. 
In Section \ref{biLip-section}, we generalize the construction and prove our main theorem for $n \ge 2$. 
We treat the case $n=1$ separately because in this case, the function $f$ serving as the body of the embedding is an explicit polynomial and there are no mixed partial derivatives to deal with. 
Also, the proof that the embedding lacks a Lipschitz extension will be simpler.

%%%%%%%%%%%%%%%%%%%%%%%%%%%%%%%%%%%%%%%%%%%%%%%%%%%
%%%%%%%%%%%%%%%%%%%%%%%%%%%%%%%%%%%%%%%%%%%%%%%%%%%
%%%%%%%%%%%%%%%%%%%%%%%%%%%%%%%%%%%%%%%%%%%%%%%%%%%

\section*{Acknowledgements}

The author deeply thanks Jeremy Tyson for many hours of discussion on the content and presentation of this paper. 
The author is grateful to Ilya Kapovich for discussion shared about metric trees. 
The author thanks the reviewers for their comprehensive advice on improving the exposition and content of this paper. 
Most notably, the reviewers stated a much simpler proof for the lack of a Lipschitz extension for $n=1$ (which the author included), suggested the use of a compactness argument to help prove that the embedding is co-Lipschitz, found an error in the original proof for the lack of a Lipschitz extension for $n\ge 2$. 
The author also thanks the reviewers for suggesting the terminology ``co-Lipschitz'' to describe a bijective map with Lipschitz inverse. 
Finally, the author thanks his wife, Alyssa Loving Jung, for all of her support.

%%%%%%%%%%%%%%%%%%%%%%%%%%%%%%%%%%%%%%%%%%%%%%%%%%%
%%%%%%%%%%%%%%%%%%%%%%%%%%%%%%%%%%%%%%%%%%%%%%%%%%%
%%%%%%%%%%%%%%%%%%%%%%%%%%%%%%%%%%%%%%%%%%%%%%%%%%%

\section{Background}\label{background-section}

\subsection{Carnot groups as metric spaces}\label{metric-space-subsection}

A Lie algebra $\mathfrak{g}$ is said to admit an \textbf{$r$-step stratification} if 
\[
	\mathfrak{g} = \mathfrak{g}_1 \oplus\cdots \oplus \mathfrak{g}_r,
\]
where $\mathfrak{g}_1\subset \mathfrak{g}$ is a subspace, $\mathfrak{g}_{j+1}= [\mathfrak{g}_j, \mathfrak{g}_1]$ for $j=1,\ldots , r-1$, and $[\mathfrak{g}_r , \mathfrak{g}] = 0$. 
We call $\mathfrak{g}_1$ the \textbf{horizontal layer} of $\mathfrak{g}$. 
A \textbf{Carnot group} is a connected, simply connected, nilpotent Lie group with stratified Lie algebra. We say that a Carnot group is step $r$ if its Lie algebra is step $r$. 

A Carnot group may be identified (isomorphically) with a Euclidean space equipped with an operation  via coordinates of the first or second kind (see Section 2 of \cite{J:AV} for more detail). 
Henceforth, we will consider Carnot groups of the form $(\R^n,\star)$. 

Let $\{X^1,\ldots , X^{m_1}\}$ be a left-invariant frame for $Lie(\R^n,\star)$. 
The horizontal bundle $H(\R^n,\star)$ is defined fiberwise by
\[
	H_p(\R^n,\star) := \text{span} \{X_p^1,\ldots  , X_p^{m_1}\}. 
\]
A path $\gamma:[a,b]\to (\R^n,\star)$ is said to be \textbf{horizontal} if it is absolutely continuous as a map into $\R^n$ and satisfies $\gamma'(t) \in H_{\gamma(t)} (\R^n,\star)$ for a.e. $t\in [a,b]$.
The \textbf{length} of a horizontal path $\gamma:[a,b]\to (\R^n,\star)$ is defined by
\[
	l(\gamma) := \int_a^b |\gamma'(t)|_H \ dt ,
\]
where $|\cdot |_H$ is induced by declaring $\{X_p^1,\ldots  , X_p^{m_1}\}$ to be orthonormal.
 
Chow proved that every Carnot group is horizontally path-connected \cite{C:US}. 
Hence, we may define a \textbf{Carnot-Carath\'eodory metric} on $(\R^n,\star)$ by 
\[
	d_{cc}(p,q) := \inf_{\gamma:[a,b]\to (\R^n,\star)} \{ l(\gamma): \gamma \text{ is horizontal}, \gamma(a) = p, \ \gamma(b) = q\}.
\]
This forms a left-invariant, geodesic metric that is one-homogeneous with respect to the group's dilations. 
We will postpone discussion of these dilations to when we discuss jet space Carnot groups. 

It is natural to wonder how the Euclidean metric structure compares with the metric structure induced by the CC-metric. 
Nagel, Stein, and Wainger proved the remarkable fact that if $(\R^n,\star)$ is a step $r$ Carnot group, then the identity map $\text{id}:\R^n\to (\R^n,\star)$ is locally $\frac{1}{r}$-H\"older while the identity map $\text{id}:(\R^n,\star)\to \R^n$ is locally Lipschitz \cite[Proposition 1.1]{NSW:BAM}. 
Not only does this imply that $\R^n$ and $(\R^n,\star)$ share the same topology, it also allows one to estimate CC-distances between points by their coordinates through the Ball-Box Theorem. 
We will delay discussion of this theorem until we discuss the metric structure of jet space Carnot groups. 

%%%%%%%%%%%%%%%%%%%%%%%%%%%%%%%%%%%%%%%%%%%%%%%%%%%
%%%%%%%%%%%%%%%%%%%%%%%%%%%%%%%%%%%%%%%%%%%%%%%%%%%
%%%%%%%%%%%%%%%%%%%%%%%%%%%%%%%%%%%%%%%%%%%%%%%%%%%

\subsection{Jet spaces as Carnot groups}

We now recall the notation of jet space Carnot groups, following Section 3 of \cite{W:JS}.

Fix $k, n\ge 1$.
Given $x_0\in \R^n$ and $f\in C^k(\R^n)$, the $k^{th}$-order Taylor polynomial of $f$ at $x_0$ is given by
\[
	T_{x_0}^k(f) = \sum_{j=0}^k \sum_{I\in I(j)} \frac{\partial_I f(x_0)}{I!} (x-x_0)^I,
\]
where $I(j)$ denotes the set of $j$-indices $(i_1,\ldots , i_n)$ ($i_1 + \cdots + i_n = j$). 
For a convenient shorthand, \emph{we  write $\tilde{I}(j) := I(0)\cup \cdots \cup I(j)$, the set of all indices of length at most $j$.}

Given $x_0\in \R^n$, we can define an equivalence relation $\sim_{x_0}$ on $C^k(\R^n)$ by $f\sim_{x_0} g $ if $T_{x_0}^k(f) = T_{x_0}^k(g)$.
We call $[f]_{\sim_{x_0}}$ the $\mathbf{k}$\textbf{-jet} of $f$ at $x_0$ and denote it by $j_{x_0}^k(f)$. 
We  then define the jet space $J^k(\R^n)$ by
\[
	J^k(\R^n) := \bigcup_{x_0\in \R^n} C^k(\R^n)/_{\sim_{x_0}}.
\]

Define 
\[
	p:J^k(\R^n)\to \R^n, \quad p(j_{x_0}^k(f)) = x_0
\]
and 
\[
	u_I:J^k(\R^n) \to \R, \quad u_I (j_{x_0}^k(f)) := \partial_I f(x_0)
\]
for $I \in \tilde{I}(k).$
We have a global chart
\[
	\psi: J^k(\R^n) \to \R^n \times \R^{d(n,k)} \times \R^{d(n,k-1)} \times \cdots \times \R^{d(n,0)}
\]
given by $\psi = (p,u^{(k)})$, where
\[
	u^{(k)} := \{ u_I : I \in \tilde{I}(k) \}.  
\]
Here,  $d(n,j)  =\binom{n+j-1}{j}$ denotes the number of distinct $j$-indices over $n$ coordinates.

For all $f\in C^k(\R^n)$ and $I\in\tilde{I}(k-1)$, 
\[
	d (\partial_I f) = \sum_{j=1}^n \partial_{I+ e_j} f\cdot  dx^j. 
\]
This motivates us to define the $1$-forms 
\[
	\omega_I := du_I - \sum_{j=1}^n  u_{I+e_j} dx^j, \quad I \in \tilde{I}(k-1) 
\]
to serve as contact forms for $J^k(\R^n)$ (see Section 3.2 of \cite{W:JS} for more detail).
The horizontal bundle of $J^k(\R^n)$ is defined  by
\[
	HJ^k(\R^n) :=  \bigcap_{I\in \tilde{I}(k-1)} \ker \omega_I. 
\]
A global frame for $HJ^k(\R^n)$ is given by
\[
	\left\{ X_j^{(k)} : j = 1,\ldots , n\right\} \cup \left\{ \frac{\partial}{\partial u_I} : I \in I(k)\right\},
\]
where 
\[
	X_j ^{(k)}:= \frac{\partial}{\partial x_j} + \sum_{I\in \tilde{I}(k-1)} u_{I+e_j} \frac{\partial}{\partial u_I}, \quad j = 1,\ldots , n.
\]
We can extend this to a global frame of $TJ^k(\R^n)$ by including  $\frac{\partial}{\partial u_I}$ for $I \in \tilde{I}(k-1)$. 
With respect to the group operation on $J^k(\R^n)$ (to be defined soon), this frame is left-invariant.

The nontrivial commutator relations are given by
\[
	\left[  \frac{\partial}{\partial u_{I+e_j}} , X_j^{(k)}\right] = \frac{\partial}{\partial u_I}, \quad I\in \tilde{I}(k-1).
\]
Thus, $Lie(J^k(\R^n))$ admits a $(k+1)$-step stratification
\[
	Lie(J^k(\R^n)) = HJ^k(\R^n)  \oplus \left\langle\frac{\partial}{\partial u_I} : I \in I(k-1)\right\rangle\oplus \cdots \oplus \left\langle \frac{\partial}{\partial u_0}\right\rangle. 
\]

One  defines a group operation on $J^k(\R^n)$ by 
\[
	(x,u^{(k)}) \odot (y,v^{(k)}) = (x+y,uv^{(k)}),
\]
where 
\[
	uv_I := v_I + \sum_{I\le J } u_J \frac{y^{J-I}}{(J-I)!}, \quad I \in \tilde{I}(k).
\]
Here, we say $I\le J$ if $I_r \le J_r$ for all $r = 1,\ldots , n$.

We will now make jet spaces more grounded by explicitly writing out the Carnot group structure of the model filiform jet spaces $J^k(\R)$. 
The $k$-jet of $f\in C^k(\R)$ at a point $x_0$ is given by
\[
 j_{x_0}^k(f) = (x_0,f^{(k)}(x_0) , \ldots , f(x_0)). 
\]
The horizontal bundle $HJ^k(\R)$ is defined by the contact forms 
\[
	\omega_j := du_j - u_{j+1} dx , \quad j =0 ,\ldots , k-1,
\]
and is framed by the left-invariant vector fields  
$X^{(k)}:= \frac{\partial}{\partial x} + u_k \frac{\partial}{\partial u_{k-1}} + \cdots + u_1 \frac{\partial}{\partial u_0}$ and $\frac{\partial}{\partial u_k}$. 
A $(k+1)$-step stratification of $\text{Lie}(J^k(\R))$ is given by
\[
	\text{Lie}(J^k(\R)) := \left\langle X^{(k)} , \frac{\partial}{\partial u_k}\right\rangle \oplus \left\langle \frac{\partial}{\partial u_{k-1}}\right\rangle \oplus \cdots \oplus \left\langle \frac{\partial}{\partial u_0}\right\rangle. 
\]
The group operation on $J^k(\R)$ is given by
\[
	(x,u_k,\ldots , u_0) \odot (y,v_k,\ldots , v_0) = (z,w_k,\ldots , w_0),
\]
where $z=x+y$, $w_k = u_k+v_k$, and 
\[
	w_s  = u_s + v_s + \sum_{j=s+1} ^k u_j \frac{y^{j-s}} {(j-s)!}, \quad s=0,\ldots ,k-1. 
\]
Despite the much simpler appearance of $J^k(\R)$ relative to that of  $J^k(\R^n)$, $n\ge 2$, valuable intuition and methods can often be built up in the model filiform case, which can later be employed for higher dimensions.

%%%%%%%%%%%%%%%%%%%%%%%%%%%%%%%%%%%%%%%%%%%%%%%%%%%
%%%%%%%%%%%%%%%%%%%%%%%%%%%%%%%%%%%%%%%%%%%%%%%%%%%
%%%%%%%%%%%%%%%%%%%%%%%%%%%%%%%%%%%%%%%%%%%%%%%%%%%

\subsection{Jet space Carnot groups as metric spaces}

We expound on Subsection \ref{metric-space-subsection} for the special case of jet space Carnot groups.

For $\epsilon>0$, define the \textbf{dilation} $\delta_\epsilon :J^k(\R^n)\to J^k(\R^n)$ by
\[
x (\delta_\epsilon j_{x_0}^k(f)) = \epsilon x_0
\] 
and
\[
	u_I (\delta_\epsilon j_{x_0}^k(f)) := \epsilon^{k+1 - |I|} \partial_I f(x_0), \quad I \in \tilde{I}(k). 
\]
In the special case  $n=1$, these dilations take the form
\[
	\delta_\epsilon (x,u_k ,u_{k-1} ,\ldots , u_0) = (\epsilon x , \epsilon u_k , \epsilon^2 u_{k-1}, \ldots , \epsilon^{k+1} u_0). 
\]
As noted before, the CC-metric is one-homogeneous with respect to these dilations: 
\[
	d_{cc} (\delta_\epsilon j_{x_0}^k(f), \delta_\epsilon j_{y_0}^k(g)) = \epsilon\cdot d_{cc} (j_{x_0}^k(f), j_{y_0} (g)). 
\]

The result of Nagel, Stein, and Wainger \cite{NSW:BAM} allows us to estimate distances in jet spaces from the algebraic structure.
\begin{theorem}\label{theorem-bb}
{(Ball-Box Theorem for jet space Carnot groups)}
Fix $k,n\ge 1$.
For $\epsilon>0$ and $p\in J^k(\R^n)$, define
\[
	Box(\epsilon) := [-\epsilon,\epsilon]^{n+d(n,k)} \times \prod_{j=2}^{k+1} [-\epsilon^j,\epsilon^j]^{d(n,k+1-j)}
\]
and
\[
	B_{cc} (p,\epsilon) := \{q \in J^k(\R^n): d_{cc}(p,q) \le \epsilon\}.
\]
There exists $C>0$ such that for all $\epsilon>0$ and $p\in J^k(\R^n)$,
\[
	B_{cc}(p,\epsilon/C) \subseteq p \odot Box(\epsilon) \subseteq B_{cc} (p,C\epsilon).
\]
\end{theorem}

From the Ball-Box Theorem, we obtain an important corollary which will serve as our most important tool for showing that our embeddings are biLipschitz. 
\begin{corollary}\label{cor-bb}
	Fix $k,n\ge 2$. 
There exists $C>0$ such that for all $(x,u^{(k)}) \in J^k(\R^n)$,
\[
	\frac{1}{C} \cdot d_{cc}(0,(x,u^{(k)})) \le \max \{ |x|, \ |u_I|^{1/(k+1-|I|)}: I\in \tilde{I}(k)\} \le C \cdot d_{cc}(0,(x,u^{(k)})). 
\]
\end{corollary} 

We will also need an observation from Rigot and Wenger \cite{RW:LNE}. This will be  key to constructing Lipschitz mappings from spheres into jet spaces.
As it is so important, and for the purposes of keeping this paper more self-contained, we will conclude this section by going over its proof.

\begin{proposition}\label{Lip-prop}
\cite[pages 4-5]{RW:LNE}
Fix $f \in C^{k+1}(\R^n).$
For all $x,y \in \R^n$,
\[
	d_{cc}(j_x^k(f), j_y^k(f)) \le \sup_{t\in [0,1]}  \left( 1 + \sum_{I\in I(k)} \sum_{j=1}^n (\partial_{I+e_j} f( x+ t(y-x)) ^2 \right)^{1/2} ||y-x||.
\]
In particular, $j^k(f) :\R^n \to J^k(\R^n)$ is locally Lipschitz.
\end{proposition}

\begin{proof}
	For $f\in C^{k+1}(\R^n)$, the jet map $j^k(f)$ is $C^1$ and horizontal with
\[
	\partial_{x_j} (j^k_x(f)) = X_j^{(k)}(j_x^k(f)) + \sum_{I\in I(k)}\partial_{I+e_j} f(x) \cdot\frac{\partial}{\partial u_I}.  
\]
For $x,y \in \R^n$, define $\gamma:[0,1] \to \R^n$, $\gamma(t) := x + t(y-x)$, to be  the straight line path connecting $x$ to $y$.
The chain rule implies $j^k(f) \circ \gamma$ is a horizontal path connecting $j_x^k(f)$ to $j_y^k(f)$.
Hence, by the definition of the CC-metric, 
\[
	d_{cc}(j_x^k(f) ,j_y^k(f)) \le  \sup_{t\in [0,1]}  \left( 1 + \sum_{I\in I(k)} \sum_{j=1}^n (\partial_{I+e_j} f( x+ t(y-x)) ^2 \right)^{1/2} ||y-x||. 
\]

As $f\in C^{k+1}(\R^n)$, $\partial_{I+e_j} f$ is bounded on compact sets for each $I\in I(k)$ and $j= 1,\ldots , n$.
It follows that the restriction of $j^k(f)$ to each compact set is Lipschitz. 
\end{proof}

%%%%%%%%%%%%%%%%%%%%%%%%%%%%%%%%%%%%%%%%%%%%%%%%%%%

%%%%%%%%%%%%%%%%%%%%%%%%%%%%%%%%%%%%%%%%%%%%%%%%%%%

\section{Embedding of the circle into $J^k(\R)$}\label{circle-section}

We  begin this section by constructing a biLipschitz embedding of  $\sS^1$ into $J^k(\R)$.
The main idea of the proof is to view $\sS^1$ as two copies of the interval $[0,\pi]$ and then apply Proposition \ref{Lip-prop} to a function with a $k^{th}$-derivative that is approximately linear near $0$ and $\pi$.

%%%%%%%%%%%%%%%%%%%%%%%%%%%%%%%%%%%%%%%%%%%%%%%%%%%
%%%%%%%%%%%%%%%%%%%%%%%%%%%%%%%%%%%%%%%%%%%%%%%%%%%
%%%%%%%%%%%%%%%%%%%%%%%%%%%%%%%%%%%%%%%%%%%%%%%%%%%

\subsection{BiLipschitz embedding  $\sS^1 \hookrightarrow J^k(\R)$}

%For our embeddings, we will take jets of smooth functions. 
%We choose the following polynomials because the behavior of their $k^{th}$-order derivatives  will allow us to later show that the embeddings are biLipschitz. 

\begin{definition}
	Fix $k\ge 1$. Define the polynomial $f_k:\R\to \R$ by $f_k(\theta) := \theta^{k+1}(\pi-\theta)^{k+1}$. 
\end{definition} 

As $f_k$ is smooth on $\R$, Proposition \ref{Lip-prop} implies that $j^k(f_k):[0,\pi]\to J^k(\R)$ is Lipschitz. 
In addition, as 
\[
	j^k_\theta (f_k)^{-1} \odot j^k_\eta (f_k) = (\eta -\theta, f_k^{(k)}(\eta ) -f_k^{(k)}(\theta),\ldots ),
\]
Corollary \ref{cor-bb} and left-invariance of $d_{cc}$ imply
\[
	|\eta -\theta|\lesssim d_{cc}(0,j^k_\theta (f_k)^{-1} \odot j^k_\eta (f_k)) =  d_{cc}(j^k_\theta(f_k) , j^k_\eta(f_k)).
\]
Here, we write $\lesssim$ to denote that the left quantity is bounded above by the right quantity up to a positive factor depending only on $k$. 

We have proven
\begin{lemma}\label{biLip-halves}
	The map $j^k(f_k) : [0,\pi ] \to J^k(\R)$ is biLipschitz. 
\end{lemma}

Gluing together two copies of $[0,\pi]$ at the endpoints, we can construct a continuous map of $\mathbb{S}^1$ into $J^k(\R)$.

\begin{definition}\label{circle-embedding-def}
Define $\phi:\sS^1\to J^k(\R)$ by
\[
	\phi(e^{i\theta}) := \left\{\begin{array}{cl}
j_\theta^k(f_k) &\quad\text{if } 0\le \theta \le \pi\\
j_{2\pi - \theta}^k(-f_k) &\quad\text{if } \pi \le \theta \le 2\pi.
\end{array}\right.
\]
\end{definition}

This map is well-defined  because $f_k^{(j)} (0) = f_k^{(j)}(\pi) = 0$ for $j=0,\ldots , k$.
A more intuitive expression of $\phi$ (which matches the original definition) is 
$\phi(e^{i\theta}) = j_\theta^k(f_k)$ and $\phi(e^{-i\theta}) = j_{\theta}^k(-f_k)$ for $0\le \theta\le \pi$. 
In this subsection, we will prove:

 \begin{theorem}\label{biLip-embed-1}
The map $\phi:\sS^1\to J^k(\R)$ is a biLipschitz embedding.
\end{theorem}

Denote the upper and lower semicircles by $\sS^1_+ := \{e^{i\theta} : 0\le \theta\le \pi\}$ and $\sS^1_-:= \{e^{i\theta}:\pi\le \theta \le 2\pi\}$, respectively. 
As  $e^{i\theta} : [0,\pi]\to \sS^1_+$  and $e^{-i\theta}:[0,\pi]\to \sS^1_- $ are biLipschitz, the restrictions $\phi|_{\sS^1_+}$ and $\phi|_{\sS^1_-}$ are biLipschitz. 
It remains to prove that 
\[
	d_{cc}(\phi(e^{i\theta}) , \phi(e^{i\eta}) ) \approx d_{\sS^1} (e^{i\theta},e^{i\eta}) \quad\text{for } e^{i\theta } \in \sS^1_+, \ e^{i\eta} \in \sS^1_-. 
\]
By $d_{\sS^1}$, we mean the geodesic path metric on $\sS^1$.
We write $A\approx B$ to denote that there exists a single constant $C$ such that 
\[
	\frac{1}{C}\cdot A \le B \le C\cdot A,
\]
for all relevant choices of $A$ and $B$. 
\emph{We will use this notation throughout this paper.}
Note that since we are merely showing maps are biLipschitz and not caring about the actual Lipschitz constants, we can allow for positive constant factors in our comparisons.

Proving that $\phi$ is Lipschitz  follows easily from the triangle inequality combined with the fact that $\phi$ is biLipschitz when restricted to the upper and lower semicircles.
Indeed, if the geodesic connecting $e^{i\theta} \in \sS^1_+$ to $e^{i\eta}\in \sS^1_-$ passes through $e^{i0}$, then 
\begin{align*}
	d_{cc}(\phi(e^{i\theta}) , \phi(e^{i\eta}) ) 
&\le d_{cc}(\phi(e^{i\theta}) , \phi(e^{i 0}) )  + d_{cc}(\phi(e^{i0}) , \phi(e^{i\eta}) ) \\
& \approx d_{\sS^1} (e^{i\theta},e^{i0}) + d_{\sS^1} (e^{i0}, e^{i\eta}) \\
& = d_{\sS^1} (e^{i\theta},e^{i\eta}).
\end{align*}
The same reasoning works  if the geodesic passes through $e^{i\pi}$. 
We have shown
\begin{proposition}\label{lip-1}
	$\phi:\sS^1\to J^k(\R)$ is Lipschitz. 
\end{proposition}

We are now halfway towards proving that $\phi$ is biLipschitz. 

\begin{definition}\label{co-Lipschitz-def}
	A map $g:X\to Y$ between metric spaces is said to be \textbf{co-Lipschitz} if there exists a constant $C>0$ such that 
\[
	d_Y (g(x_1),g(x_2)) \ge \frac{1}{C} \cdot d_X(x_1,x_2)\quad \text{for all } x_1,x_2\in X.
\]
If a map is co-Lipschitz, we say it has the \textbf{co-Lipschitz property}. 
\end{definition} 

It remains to show that $\phi$ is co-Lipschitz.
Before we prove this, we will observe that the $k^{th}$ derivative of $f_k$ is approximately linear near $0$ and near $\pi$. 
This behavior was the primary reason for our choice of $f_k$.

\begin{lemma}\label{linear-lemma}
There exists  a constant $0<\epsilon< 1$ such that
\[
	f_k^{(k)}(\theta) \ge \frac{\pi^{k+1}(k+1)! }{2}  \cdot \theta \quad\text{if }  0\le \theta \le \epsilon
\]
and 
\begin{align*}
	\left\{\begin{array}{cc}
 f_k^{(k)}(\theta) \ge \frac{\pi^{k+1}(k+1)!}{2}\cdot (\pi - \theta) &\quad\text{if }  \pi-\epsilon\le \theta \le \pi \text{ and } k \text{ is even}\\
& \\
f_k^{(k)}(\theta) \le - \frac{\pi^{k+1}(k+1)!}{2} \cdot (\pi -\theta) & \quad\text{if } \pi-\epsilon \le\theta\le \pi \text{ and } k \text{ is odd.}
\end{array}\right.
\end{align*}
\end{lemma}

\begin{proof}
	By induction, 
\[
	f_k^{(k)} (\theta) = (k+1)! \theta (\pi-\theta)^{k+1} + \theta^2 p(\theta)
\]
and
\[
	f_k^{(k)} (\theta) = (k+1)! (-1)^k \theta^{k+1} (\pi-\theta)   + (\pi-\theta)^2 q(\theta),
\]
for some polynomials $p,$ $ q$. 
This implies 
\[
	\lim_{\theta\to 0}\frac{f_k^{(k)}(\theta)}{\theta} = 
	\lim_{\theta \to\pi} \frac{(-1)^k f_k^{(k)}(\theta)}{\pi - \theta} = (k+1)! \cdot  \pi^{k+1}.
\]
The lemma follows. 
\end{proof}

We can now finish the proof of Theorem \ref{biLip-embed-1}, proving that $\phi$ is biLipschitz.

\begin{proof}[Proof of Theorem \ref{biLip-embed-1}]
We proved in Proposition \ref{lip-1} that $\phi$ is Lipschitz. It remains to show $\phi$ is co-Lipschitz, i.e., that  there exists a constant $C>0$ such that 
\[
	d_{cc}(\phi(e^{i\theta}),\phi( e^{-i\eta})) \ge \frac{1}{C} \cdot d_{\mathbb{S}^1} (e^{i\theta},e^{-i\eta})  
\]
for all $e^{i\theta}\in \mathbb{S}_+^1$ and $e^{-i\eta} \in \mathbb{S}_-^1$.

Let $0<\epsilon<1$ be the constant from Lemma \ref{linear-lemma}.
To prove the co-Lipschitz property, it suffices to consider three arrangements of pairs of points $e^{i\theta} \in \mathbb{S}^1_+$ and $e^{-i\eta} \in \mathbb{S}^1_-$, where $0\le \theta, \eta \le \pi$:

\begin{enumerate}[label=(\roman*)]
\item $0 \le \theta,\eta \le \epsilon$, or $\pi-\epsilon \le \theta,\eta\le \pi$ (points are close to each other and the $x$-axis).
\item $\epsilon \le \theta\le \pi-\epsilon$ or $\epsilon\le \eta \le \pi-\epsilon$ (one of the points is far from the $x$-axis).
\item $|\theta - \eta| \ge \pi -2\epsilon$ (arguments are  far from each other). 
\end{enumerate}
(Readers should convince themselves that these cases handle all possible pairs of a point on the upper semicircle and a point on the lower semicircle.)

Case (i):  Fix $0\le \theta,\eta \le\epsilon$. 
By Corollary \ref{cor-bb} and Lemma \ref{linear-lemma}, 
\begin{align*}
	d_{cc}(\phi (e^{i\theta}) , \phi(e^{-i\eta}))&= d_{cc}(j_\theta^k(f_k), j_\eta^k(-f_k))\\
& \gtrsim |f_k^{(k)} (\theta) + f_k^{(k)}(\eta)|\\
& \ge \frac{(k+1)!}{2} \cdot (\theta + \eta)\\
& = \frac{(k+1)!}{2} \cdot d_{\sS^1} (e^{i\theta},e^{-i\eta}). 
\end{align*}
 A similar calculation shows 
\begin{align*}
	d_{cc}(\phi(e^{i\theta} ) ,\phi(e^{-i\eta})) \gtrsim \frac{(k+1)!}{2} \cdot (2\pi - \theta -\eta) = \frac{(k+1)!}{2} \cdot d_{\sS^1} (e^{i\theta},e^{i\eta})
\end{align*} 
for  $\pi-\epsilon \le\theta,\eta\le \pi$.
This handles case (i). 

Case (ii): Suppose $\epsilon \le \theta \le \pi-\epsilon$ and $0\le \eta\le \pi$. 
Then $f_k(\theta) >0$ while $-f_k(\eta) \le 0$.
Hence, $j_\theta^k(f_k) \ne j_\eta^k(-f_k)$, so that
\[
	0 < d_{cc}(j_\theta^k(f_k), j_\eta^k(-f_k)) = d_{cc}(\phi(e^{i\theta} ) , \phi(e^{-i\eta})). 
\]
This implies that the restriction of $d_{cc}$ on the compact set 
\[
	\{\phi(e^{i\theta}): \epsilon \le \theta\le \pi-\theta \} \times \{\phi(e^{-i\eta}): 0 \le \eta \le \pi\}
\]
is strictly positive. 
By the Extreme Value Theorem, there must exist $\delta_1>0$ such that
\[
	d_{cc}(\phi(e^{i\theta} ) , \phi(e^{-i\eta})) >\delta_1
\]
whenever $\epsilon \le \theta \le \pi-\epsilon$ and $0\le \eta\le \pi$. 
By the same argument, there also exists $\delta_2>0$ such that 
\[
	d_{cc}(\phi(e^{i\theta} ) , \phi(e^{-i\eta})) >\delta_2
\]
whenever $0 \le \theta \le \pi$ and $\epsilon\le \eta\le \pi-\epsilon$.
As $\mathbb{S}^1$ is bounded, this handles case (ii).

Case (iii): This case is handled in the same way as case (ii) was. We need only observe that $\{(e^{i\theta}, e^{-i\eta}) \in \sS^1_+\times \sS^1_- : |\theta-\eta| \ge \pi - 2\epsilon, \ 0\le \theta,\eta\le \pi\}$ is compact and $j_\theta^k(f_k) \ne j_\eta^k(-f_k)$ whenever $\theta\ne \eta$. 

This concludes the proof that $\phi$ is co-Lipschitz, hence biLipschitz. 
\end{proof}

%%%%%%%%%%%%%%%%%%%%%%%%%%%%%%%%%%%%%%%%%%%%%%%%%%%
%%%%%%%%%%%%%%%%%%%%%%%%%%%%%%%%%%%%%%%%%%%%%%%%%%%
%%%%%%%%%%%%%%%%%%%%%%%%%%%%%%%%%%%%%%%%%%%%%%%%%%%

\subsection{The embedding does not admit a Lipschitz extension and $\pi_m^{Lip}(J^k(\R)) = 0$}

In this section, we will prove that the embedding from Theorem \ref{biLip-embed-1} does not admit a Lipschitz extension. 
The author originally proved this by modifying an argument of Haj{\l}asz, Schikorra and Tyson for $\mathbb{H}^1$ \cite{HST:HG}. 
Then a reviewer provided a much simpler, clearer proof.
The author wants to reiterate his appreciation to the reviewer for this. 
We will also prove that each of the Lipschitz homotopy groups of $J^k(\R)$ is trivial. 
These proofs will rely on a result of Wenger and Young \cite[Theorem 5]{WY:LHG}, which states in particular that every Lipschitz map from $\mathbb{B}^2$ to $J^k(\R)$ factors through a metric tree.

%To prove the triviality of higher Lipschitz homotopy groups of $\mathbb{H}^1$, 
In   \cite{WY:LHG},  Wenger and Young prove that every Lipschitz mapping from $\sS^m$, $m\ge 2$,  to $\mathbb{H}^1$ factors through a metric tree.
A \emph{metric tree} (or $\R$-tree) is a geodesic metric space for which every geodesic triangle is isometric to a tripod, or equivalently, is $0$-hyperbolic in the sense of Gromov. 
Metric trees are CAT($\kappa$) spaces for all $\kappa \le 0$ and are uniquely geodesic (see Proposition 1.4(1) and Example 1.15(5) of Chapter II.1 in \cite{BH:MS}). 
We note that in his book on the more general $\Lambda$-trees \cite{C:IT}, Chiswell defines metric trees in a manner equivalent to as above (see Lemmata 2.1.6 and 2.4.13 of \cite{C:IT}). 
For a much greater discussion on metric trees, we refer the reader to this book  \cite{C:IT}. 
The first property of metric trees below is usually cited without proof while the second was stated without proof in \cite{WY:LHG}.
We will provide justification here.

\begin{lemma}\label{completion-metric-tree}
For every  metric tree $(Z,d)$, its completion  $(\hat{Z},\hat{d})$  is a metric tree and is Lipschitz contractible. 
\end{lemma}

\begin{proof}
Let $(Z,d)$ be a metric tree. 
Chiswell proved that the completion of a metric tree $(\hat{Z},\hat{d})$ is still a metric tree \cite[Theorem 2.4.14]{C:IT} (we note that this result is usually attributed to Imrich at \cite{I:OMP}, but the author was unable to track down this work). 
%Chiswell proved that every connected metric  space that is $0$-hyperbolic (in the sense of Gromov),  is a metric tree \cite[Lemma 4.13]{C:IT}. 
%This implies that the completion $(\hat{Z},\hat{d})$ of $Z$ is a metric tree (this fact is usually attributed to Imrich at \cite{I:OMP}, but his work could not be tracked down by the author). 
Then since metric trees are CAT($\kappa$) spaces for all $\kappa\le 0$, a version of Kirszbraun's theorem proven by Lang and Schroeder \cite[Theorem B]{LS:K}  implies that $(\hat{Z},\hat{d})$ is Lipschitz contractible. 
\end{proof} 

%Lang and Schroeder proved a variant of Kirszbraun's theorem which implies metric trees are Lipschitz contractible (\cite[Theorem B]{LS:K}). 
A metric space $X$ is \emph{quasi-convex} if there exists a constant $C$ such that every two points $x,y\in X$ can be connected by a path of length at most $Cd(x,y)$. 
For example, each sphere $\sS^n$ is quasi-convex. 
In 2014, Wenger and Young proved a factorization result for mappings into purely $2$-unrectifiable spaces.

\begin{theorem}\label{WY-theorem}
\cite[Theorem 5]{WY:LHG}
Let $X$ be a quasi-convex metric space with $\pi_1^{Lip}(X) =0$. 
Let furthermore $Y$ be a purely $2$-unrectifiable metric space.
Then every Lipschitz map from $X$ to $Y$ factors through a metric tree.
That is, there exist a metric tree $Z$ and Lipschitz maps $\varphi:X\to Z$ and $\psi: Z\to Y$ such that $f= \psi \circ \varphi$. 
\end{theorem}

Wenger and Young used this result to prove that $\pi_m^{Lip}(\mathbb{H}^1) =0$ for all $m\ge 2$ \cite[Corollary 4]{WY:LHG}.
We can easily modify their proof to prove the triviality of $\pi_m^{Lip}(J^k(\R))$ for $m\ge 2$.
We only include a  proof to help keep this paper self-contained. 
We note that Lipschitz homotopy groups $\pi_m^{Lip}(J^k(\R))$ are defined in the same way as typical homotopy groups are, except the maps and  homotopies are required to be Lipschitz (see Section 4 of \cite{DHLT:OTL} for a greater discussion). 

\begin{corollary}
 For $m\ge 2$ and $k\ge 1$, $\pi_m^{Lip}(J^k(\R)) = 0$. 
\end{corollary}

\begin{proof}
Fix $m\ge 2$ and $k\ge 1$. 
Suppose $f:\sS^m\to J^k(\R)$ is Lipschitz.
By a theorem of Magnani, $J^k(\R)$ is purely $2$-unrectifiable \cite[Theorem 1.1]{M:UAR}. 
Hence, by Theorem \ref{WY-theorem}, there exist a metric tree $Z$ and Lipschitz maps $\varphi:\sS^m\to Z$ and $\psi:Z\to J^k(\R)$ such that $f = \psi \circ \varphi$. 
Lemma \ref{completion-metric-tree} combined with the fact that $J^k(\R)$ is complete imply that we may assume $Z$ is complete. 
Furthermore, by Lemma \ref{completion-metric-tree}, there exists a Lipschitz homotopy $h:Z\times [0,1]\to Z$ of the identity map to a constant map. 
Then $\alpha: \sS^m\times [0,1]\to J^k(\R)$ defined by $\alpha(x,t) =( \psi \circ h)(\varphi(x),t) $ is a Lipschitz homotopy of $f$ to a constant map. 
\end{proof}

\begin{proof}[Proof of Theorem \ref{main-theorem} for $n=1$] 
Suppose, for contradiction, that the biLipschitz embedding $\phi:\mathbb{S}^1\to J^k(\R)$ from Theorem \ref{biLip-embed-1} admits a Lipschitz extension $\tilde{\phi}:\mathbb{B}^2\to J^k(\R)$. 
Since $J^k(\R)$ is purely $2$-unrectifiable, Wenger and Young's result (Theorem \ref{WY-theorem}) implies that $\tilde{\phi}$, and hence $\phi$,  factors through a metric tree.
However, any two topological embeddings of $[0,1]$ into a metric tree that share common endpoints must have the same image.
This leads to a contradiction that $\phi$ is injective. 
%The closed map lemma states that every continuous function f : X → Y from a compact space X to a Hausdorff space Y is closed and proper (i.e. preimages of compact sets are compact).
\end{proof}

%%%%%%%%%%%%%%%%%%%%%%%%%%%%
%\hrule %%%%%%%%%%%%%%%%%%%%%%%%%%%%
%%%%%%%%%%%%%%%%%%%%%%%%%%%%

%%%%%%%%%%%%%%%%%%%%%%%%%%%%%%%%%%%%%%%%%%%%%%%%%%%
%%%%%%%%%%%%%%%%%%%%%%%%%%%%%%%%%%%%%%%%%%%%%%%%%%%
%%%%%%%%%%%%%%%%%%%%%%%%%%%%%%%%%%%%%%%%%%%%%%%%%%%

\section{Embedding of  sphere into $J^k(\R^n)$}\label{biLip-section}

%%%%%%%%%%%%%%%%%%%%%%%%%%%%%%%%%%%%%%%%%%%%%%%%%%%
\begin{comment}%%%%%%%%%%%%%%%%%%%%%%%%%%%%%%%%%%%%%%%%%%%%%%

We now move to constructing biLipschitz embeddings of spheres into jet space Carnot groups. 
In the model filiform case, we realized the circle as two copies of the interval $[0,\pi]$, then took the jet of a positive polynomial with a linear $k^{th}$-order derivative.
We hope to repeat this strategy for higher dimensions.

Once we construct this embedding, we will run into the issue that Proposition \ref{rank-condition} no longer applies.
In fact, if $n\ge 2$, the embedding $f:\R^{n+1}\to J^k(\R^n)$ given by
\[
	x\mapsto (0,x,0),
\]
where the first $0$ lies in $\R^n$, is biLipschitz (see Remark \ref{bad-remark}). 
Instead, we will modify the argument of Rigot and Wenger in \cite{RW:LNE} used to prove that there exists a Lipschitz map $f:\sS^n\to J^k(\R^n)$ that doesn't admit a Lipschitz extension. 

\end{comment}%%%%%%%%%%%%%%%%%%%%%%%%%%%%%%%%%%%%%%%%%%%%%%
%%%%%%%%%%%%%%%%%%%%%%%%%%%%%%%%%%%%%%%%%%%%%%%%%%%

%%%%%%%%%%%%%%%%%%%%%%%%%%%%%%%%%%%%%%%%%%%%%%%%%%%
%%%%%%%%%%%%%%%%%%%%%%%%%%%%%%%%%%%%%%%%%%%%%%%%%%%
%%%%%%%%%%%%%%%%%%%%%%%%%%%%%%%%%%%%%%%%%%%%%%%%%%%

%\subsection{A biLipschitz embedding $\sS^n\hookrightarrow J^k(\R^n)$}

In this section, we will prove our main theorem, Theorem \ref{main-theorem},  for $n\ge 2$. 
We begin by stating the section's assumptions and  notation. 
We will assume $n\ge 2$. 
Whenever we write $|x|$, we will mean the norm of $x\in \R^n$ with respect to the standard Euclidean metric. 
On the other hand, when we are calculating distances between points and write $\rho(\cdot,\cdot)$, we will be referring to the \emph{Manhattan metric} on Euclidean space. 
Explicitly, for $x,y\in \R^n$,
\[
 	\rho(x,y) := \sum_{i=1}^n |x_i-y_i|. 
\]
In Proposition \ref{phi-Lip},  we will use the geodesic path metric on $\sS^n$ and denote it by $d_{\sS^n}(\cdot,\cdot)$.
Of course, there are no problems switching between  these three metrics since they are all  equivalent (see Theorem 3.1 of \cite{DHLT:OTL} for equivalence of path metric and Euclidean metric). 

%%%%%%%%%%%%%%%%%%%%%%%%%%%%%%%%%%%%%%%%%%%%%%%%%
%%%%%%%%%%%%%%%%%%%%%%%%%%%%%%%%%%%%%%%%%%%%%%%%%
%%%%%%%%%%%%%%%%%%%%%%%%%%%%%%%%%%%%%%%%%%%%%%%%%

\subsection{Construction of biLipschitz embedding $\sS^n\hookrightarrow J^k(\R^n)$}

For the case $n=1$, we implicitly used that the exponential $e^{i\theta}:[0,\pi]\to \sS^1$ is biLipschitz. 
This allowed us to view the upper and lower semicircles as copies of $[0,\pi]$. 
We then employed a smooth function  $f_k:[0,\pi]\to \R$ to define our biLipschitz map $\phi:\sS^1\to J^k(\R)$.
We will follow a similar strategy in higher dimensions.

We begin with some notation. 

\begin{definition}
	Define the \emph{upper hemisphere}
\[
	\sS_+^n := \{(x,t) \in \R^{n+1}= \R^n\times \R: |x|^2 +t^2 = 1, \ t\ge 0\}
\]
and the \emph{lower hemisphere} 
\[
	\sS_-^n := \{(x,t) \in \R^{n+1} = \R^n\times \R: |x|^2 + t^2 = 1, \ t\le 0\}. 
\]
\end{definition}

Note $\sS^n = \sS_+^n \cup \sS_-^n$ with $\sS_+^n \cap \sS_-^n = \sS^{n-1} \times \{0\}$. 
I will later refer to this last set as the \textbf{equator} of $\mathbb{S}^n$. 

Our first step will be to determine how to lift the $n$-ball to the upper hemisphere in a biLipschitz way.
We will accomplish this via polar coordinates. 

\begin{proposition}\label{polar-coord}
	The map $L:\B^n\to \sS^n_+$ defined by
\[
	L(\theta\cdot x) = (x\cdot \sin (\pi\theta/2) , \cos (\pi\theta /2)), \quad \theta \in [0,1], \ x\in \sS^{n-1}
\]
is well-defined and biLipschitz. 
\end{proposition}

\begin{proof}
	It isn't hard to see that $L$ is well-defined. 

%We will write elements of $\B^n$ in the form $(x,y)$, where $x\in \R$ and $y\in \R^{n-1}$. 
Via a rotation, it suffices to assume we have two points $(\eta,0)$, $\theta \cdot (x,y) \in \B^n$, where $0<\eta \le 1$, $0\le \theta \le \eta$, and $(x,y)\in \sS^{n-1} \subset  \R\times \R^{n-1}$. 

First note 
\[
	\rho_{\B^n} ((\eta,0), \theta\cdot (x,y)) = (\eta -\theta x) + \theta \sum_{i=2}^n |y_i|
\]
(recall we are using the Manhattan metric). 
We have (with justification below) 
\begin{align*}
	\rho_{\sS_+^n} \biggl( (\sin & \frac{\pi\eta}{2}, 0,\cos \frac{\pi\eta}{2}) , (x\sin \frac{\pi\theta}{2} , y \sin \frac{\pi \theta}{2} , \cos \frac{\pi\theta}{2})\biggr)\\
& =  \left( \sin \frac{\pi\eta}{2} - x\sin \frac{\pi\theta}{2} \right) + \sin \left(\frac{\pi\theta}{2}\right) \sum_{i=2}^n |y_i| +  \left( \cos \frac{\pi\theta}{2} -\cos \frac{\pi\eta}{2}\right) \\
& = \left( \sin \frac{\pi\eta}{2} - \sin \frac{\pi\theta}{2} \right) + \sin\left( \frac{\pi\theta}{2} \right)(1-x) + \sin \left(\frac{\pi\theta}{2}\right)  \sum_{i=2}^n |y_i|+  \left( \cos \frac{\pi\theta}{2} -\cos \frac{\pi\eta}{2}\right)\\
& \approx |e^{i\pi\eta/2} -e^{i\pi \theta/2}| + \theta(1-x) + \theta \sum_{i=2}^n |y_i| \\
& \approx (\eta - \theta x )+  \theta \sum_{i=2}^n |y_i|\\
& = \rho_{\B^n} ((\eta,0), \theta\cdot (x,y)) . 
\end{align*}
For the first approximation above, we used the fact that the Manhattan metric and standard Euclidean metric are uniformly equivalent.
We also used that   $\sin \theta\approx \theta$ on $[0,\pi/2]$. 
For the second approximation, we used that the Euclidean metric and geodesic path metric are uniformly equivalent on the upper half circle. 
\end{proof}

Recalling the strategy used to embed a circle, we now find a smooth function on $\R^{n}$ to serve as the ``body of our jet.''
For the circle, the main difficulty was finding a  positive function $f_k$ that satisfied 
\[
f_k^{(k)} (\theta) \approx \theta = \rho_{\sS^1}(e^{i\theta},e^{i0}) \quad \text{for } \theta \text{ near } 0
\]
and similar behavior for $\theta $ near $\pi$. 
For general  $n$, the natural choice would be $f(x) := (1-|x|)^{k+1}$.
However, $f$ has a singularity at $0$.
Fortunately, we only need $f$ to equal $(1-|x|)^{k+1}$ near the boundary of $\B^n$.
We encapsulate the necessary conditions of $f$ in the following lemma.

\begin{lemma}\label{body-of-jet}
	There exists a smooth function $f:\R^n\to \R$ satisfying:
\begin{enumerate}[label=(\alph*)]
\item $f(x) = (1-|x|)^{k+1} $ for $\frac{1}{2} \le |x|\le \frac{3}{2}$; and 
\item $f(x)>0$ for $|x|<1$.
%\item $f$ has compact support. 
\end{enumerate}
%Here, we write $|x|$ to denote the  norm of $x$ with respect to the standard Euclidean norm. 
\end{lemma}

\begin{proof}
	Choose a smooth function $\alpha:\R^n\to [0,1]$ satisfying $\alpha =1$ on $\{x:\frac{1}{2}\le |x|\le \frac{3}{2}\}$ and $\alpha =0$ on $\{x:|x| \le \frac{1}{4}\}$. 
Then $\alpha(x)\cdot (1-|x|)^{k+1}$ satisfies property (a).
To satisfy  (b) as well, we merely need to add a smooth, non-negative function that is zero on $\{x:\frac{1}{2}\le |x|\le \frac{3}{2}\}$ and is positive where $\alpha =0$ in $\B^n$. 
But $1-\alpha$ clearly satisfies these conditions.
Hence, $f:\R^n\to \R$ defined by 
\[
	f(x) := \alpha(x) \cdot (1-|x|)^{k+1} + (1-\alpha(x))
\]
works. 
\end{proof}

\begin{definition}\label{smooth-embedding-phi}
Let $f:\R^n\to \R$ be a function satisfying properties (a) and (b) of Lemma \ref{body-of-jet}.
We define $\phi:\sS^n\to J^k(\R^n)$ by
\[
	\phi(x\sin (\pi\theta/2), t) := \left\{
\begin{array}{cl}
j_{\theta\cdot x}^k(f) &\quad\text{if } x\in \sS^{n-1}, \ 0\le \theta\le 1, \ t\ge 0\\
j_{\theta\cdot x}^k(-f) &\quad\text{if } x\in \sS^{n-1}, \ 0\le \theta \le 1, \ t\le 0.
\end{array}\right.
\]
\end{definition}
Observe that $\phi$ is well-defined since $\partial_I f(x) =0$ whenever $|x|=1$ and $|I| \le k$. 
We will prove:

\begin{theorem}\label{biLip-embedding}
The map $\phi:\mathbb{S}^n\to J^k(\R^n)$ is a biLipschitz embedding. 
\end{theorem}

As in the circle case, proving that $\phi$ is Lipschitz is easier than proving that $\phi$ is co-Lipschitz (see Definition \ref{co-Lipschitz-def}), so we will do the former first.
Before this, we need to prove that $\phi$ is biLipschitz when restricted to the upper and lower hemispheres.
\begin{lemma}\label{biLip-halves}
 	The restrictions $\phi|_{\sS^n_+}$ and $\phi|_{\sS^n_-}$ are biLipschitz. 
\end{lemma}

\begin{proof}
By Proposition \ref{Lip-prop} and the Ball-Box Theorem, $j^k(f):\mathbb{B}^n \to J^k(\R^n)$ is biLipschitz. 

Let $L:\mathbb{B}^n \to \sS^n_+$ be the biLipschitz map defined in Proposition \ref{polar-coord}. 
Then  the restriction $\phi|_{\sS^n_+} = j^k(f)\circ L^{-1}$ is biLipschitz. As the reflection $R:\sS^n\to \sS^n$ given by $(x,t)\mapsto (x,-t)$, $x\in \B^{n-1}$, $t\in \R$ is an isometry, the restriction $\phi|_{\sS^n_-} = \phi|_{\sS^n_+} \circ R$ is also biLipschitz. 
\end{proof}

It remains to consider the application of  $\phi$ to points on opposite halves of $\sS^n$. 
More precisely, we need to prove
\[
	d_{cc}(j_{\eta\cdot x}^k(f) , j_{\theta\cdot y}^k(-f)) \approx \rho_{\sS^n} ((x\sin (\pi\eta/2), \cos (\pi\eta/2)), (y\sin(\pi\theta/2), -\cos (\pi\theta/2))) 
\]
for $x,y \in \sS^{n-1} , \ 0\le \eta, \theta\le 1. $

Proving that $\phi$ is Lipschitz will be proven in the same way here as it was for $n=1$ (see Proposition \ref{lip-1}).

\begin{proposition}\label{phi-Lip}
$\phi:\sS^n\to J^k(\R^n)$ is Lipschitz. 
\end{proposition}

\begin{proof}
It remains to prove
\[
	d_{cc}(j_{\eta\cdot x}^k(f) , j_{\theta\cdot y}^k(-f)) \lesssim \rho_{\sS^n} ((x\sin (\pi\eta/2), \cos (\pi\eta/2)), (y\sin(\pi\theta/2), -\cos (\pi\theta/2))) 
\]
for $x,y \in \sS^{n-1} $ and $ 0\le \eta, \theta< 1. $

Let $ (x\sin (\pi\eta/2), \cos (\pi\eta/2)) \in \mathbb{S}_+^n$, $(y\sin(\pi\theta/2), -\cos (\pi\theta/2)) \in \mathbb{S}_-^n$, $\gamma:[0,1]\to\mathbb{S}^n$ the geodesic  connecting them, and $r\in [0,1]$ such that $\gamma(r)$ is on the equator. 
Note that $j_z^k(f) = \gamma(r)= j_z^k(-f)$ if $\gamma(r) = (z,0)$. 
By Lemma \ref{biLip-halves}, 
\begin{align*}
	d_{cc}(j_{\eta\cdot x}^k&(f) , j_{\theta\cdot y}^k(-f) ) \\
&\le d_{cc} (j_{\eta\cdot k}^k(f) , j_z^k(f)) + d_{cc} (j_z^k(-f), j_{\theta\cdot y}^k(-f)) \\
&\approx \rho_{\sS^n} ((x\sin(\pi\eta/2),\cos(\pi\eta/2)),(z,0)) + d_{\sS^n} ((z,0), (y\sin(\pi\theta/2),-\sin(\pi\eta/2)))\\
&\approx d_{\sS^n} ((x\sin(\pi\eta/2),\cos(\pi\eta/2)),(z,0)) + \tilde{d}_{\sS^n} ((z,0), (y\sin(\pi\theta/2), -\sin(\pi\eta/2)))\\
&= d_{\sS^n}  ((x\sin(\pi\eta/2), \cos(\pi\eta/2)), (y\sin(\pi\theta/2),-\cos(\pi\theta/2)))\\
&\approx \rho_{\sS^n}  ((x\sin(\pi\eta/2), \cos(\pi\eta/2)), (y\sin(\pi\theta/2),-\cos(\pi\theta/2))),
\end{align*}
where $d_{\mathbb{S}^n}$ denotes the geodesic path metric on $\mathbb{S}^n$. 
\end{proof}

It remains to prove that $\phi:\sS^n\to J^k(\R^n)$ is co-Lipschitz. 
As in the initial case $n=1$, we first need to prove that certain $k^{th}$-order derivatives of $(1-|x|)^{k+1}$ are approximately linear near the boundary of $\B^n$.

\begin{lemma}\label{der-lemma}
Let $f:\R^n\to \R$ be a smooth function satisfying properties (a)-(b) of Lemma \ref{body-of-jet}. 
There exist constants $0<\epsilon <\frac{1}{2}<C$ satisfying the following:
For all $i=1,\ldots , n$ and $x\in \R^n$ satisfying $1-\epsilon \le |x|\le 1$ and $|x_i| > \frac{1}{4\sqrt{n}}$, we have
\[
	\left\{ \begin{array}{cl}
\frac{1-|x|}{C} \le \frac{\partial^kf}{\partial x_i^k}(x) \le C(1-|x|) &\quad\text{if } k \text{ is even}\\
\frac{1-|x|}{C} \le \frac{\partial^kf}{\partial x_i^k}(x) \le C(1-|x|) &\quad\text{if } k \text{ is odd and } x_i <0\\
\frac{1-|x|}{C} \le -\frac{\partial^kf}{\partial x_i^k}(x) \le C(1-|x|) &\quad\text{if } k \text{ is odd and } x_i>0.
\end{array}\right.
\]
\end{lemma}

\begin{proof}
	%Fix a $k$-index $I$ ($i_1+\cdots + i_n  = k$).
Fix $i=1,\ldots , n$. 
By condition (a), $f(x)  = (1-|x|)^{k+1}$ for $\frac{1}{2}<|x| <\frac{3}{2}$. 
We have 
\[
	\frac{\partial f}{\partial x_i} (x)  = (1-|x|)^k \cdot \frac{ - (k+1) x_i}{\sqrt{x_1^2 + \cdots + x_n^2}}  \quad \text{for } \frac{1}{2}<|x|<\frac{3}{2}.
\]
By induction,  there exists a smooth function $g_i:\{x\in \R^n:\frac{1}{2} \le |x|\le \frac{3}{2}\}\to\R$ such that
\[
	\frac{\partial^kf}{\partial x_i^k}(x) = (1-|x|) \cdot \frac{(-1)^k (k+1)! x_i^k}{(x_1^2 +\cdots + x_n^2)^{\frac{k}{2}}} + (1-|x|)^2 g_i(x) 
\]
for $\frac{1}{2} < |x|<\frac{3}{2}$. 
Restricting to $x$ with $|x_i|\ge \frac{1}{4\sqrt{n}}$, the second term becomes relatively neglible as $|x|\to 1$. 
The lemma follows. 
\end{proof}

We can now prove that $\phi:\mathbb{S}^n\to J^k(\R^n)$ is co-Lipschitz, hence biLipschitz by Proposition \ref{phi-Lip}.

\begin{proof}[Proof of Theorem \ref{biLip-embedding}]
It remains to prove that $\phi$ is co-Lipschitz, i.e., that exists a constant $D$ such that 
\[
	d_{cc} (j_{\eta\cdot x}^k(f) , j_{\theta\cdot y}^k(-f)) \ge \frac{1}{D} \cdot \rho_{\sS^n} ((x\sin (\pi\eta/2),s), (y\sin (\pi\theta/2),-t)). 
\]
for all points $(x\sin (\pi\eta/2),s) , $ $(y\sin (\pi\theta/2), -t) \in \sS^n$ with $x,y\in \sS^{n-1}$, $s,t > 0$,   and $0\le \eta,\theta\le 1$.

Let $\epsilon,$ $C$ be the constants from Lemma \ref{der-lemma}.
Consider the following three properties:
\begin{enumerate}[label=(\Alph*)]
\item $\eta\ge 1-\epsilon$.
\item $\theta\ge 1-\epsilon$. 
\item $|\eta\cdot x - \theta \cdot y | \le \frac{1}{4\sqrt{n}}$. 
\end{enumerate}

First suppose that at least one of properties (A)-(C) is not satisfied.
None of the pairs in the compact sets 
\begin{itemize}
\item $\{(\phi(x\sin (\pi\eta/2),s), \phi(y\sin (\pi\theta/2), -t))\in \mathbb{S}^n \times \mathbb{S}^n: s,t \ge 0, \ 1-\epsilon \le \eta \le 1, \ 0 \le \theta \le 1\}$;
\item $\{(\phi(x\sin (\pi\eta/2),s), \phi(y\sin (\pi\theta/2), -t))\in \mathbb{S}^n \times \mathbb{S}^n: s,t \ge 0, \ 1-\epsilon \le \theta \le 1, \ 0 \le \eta \le 1\}$;
\item $\{(\phi(x\sin (\pi\eta/2),s), \phi(y\sin (\pi\theta/2), -t))\in \mathbb{S}^n \times \mathbb{S}^n: s,t \ge 0, \ 0 \le \theta,\eta \le 1, \ |\eta\cdot x - \theta \cdot y | \ge \frac{1}{4\sqrt{n}}\}$
\end{itemize}
are of the form $(x,x) $ for $x\in \mathbb{S}^n$. 
By the Extreme Value Theorem, it follows that there exists $\delta>0$ such that 
\[
	d_{cc}(z_1,z_2) >\delta
\]
for each pair $(z_1,z_2)$ in the above compact sets.

Now suppose that properties (A)-(C) are satisfied. 
By Proposition \ref{polar-coord}, 
\[
	\rho_{\sS^n} ((x\sin (\pi\eta/2),s) , (y\sin (\pi\theta/2), t)) = \rho_{\sS^n}(L(\eta\cdot x), L(\theta\cdot y)) \approx |\eta\cdot x - \theta\cdot y|. 
\]
In particular,
\[
	|x\sin (\pi\eta/2) -y\sin (\pi \theta/2)| \lesssim |\eta\cdot x-\theta\cdot y|. 
\]
As $p(j_{\eta\cdot x}^k (f)^{-1}\odot j_{\theta\cdot y}^k(-f)) = \theta\cdot y-\eta\cdot x$, we then have 
\begin{equation}\label{first-bound}
	 \rho_{\mathbb{B}^n}(x\sin (\pi\eta/2),y\sin (\pi \theta/2))  \approx |x\sin (\pi\eta/2) -y\sin (\pi \theta/2)| \lesssim d_{cc} (j_{\eta\cdot x}^k (f) , j_{\theta\cdot y}^k(-f))
\end{equation}
by Corollary \ref{cor-bb}.

As  
\[
\rho_{\sS^n} ((x\sin (\pi\eta/2),s), (y\sin (\pi\theta/2),-t)) = \rho_{\mathbb{B}^n}(x\sin (\pi\eta/2),y\sin (\pi \theta/2)) + |s+t|,
\]
it remains to  bound $|s+t|$ from above by (a multiple of) $d_{cc} (j_{\eta\cdot x}^k(f) , j_{\theta\cdot y}^k(-f))$.
Note $s = \cos (\pi\eta/2) $ and $t= \cos (\pi\theta/2)$.
Via the Taylor series expansion of cosine at $\pi/2$, 
\[
	\cos \nu = \pi/2- \nu + O ((\pi/2 - \nu)^3) \quad\text{as } \nu \to \pi/2.
\]
It follows that 
\[
	\cos \frac{\pi\nu}{2}  \lesssim \frac{\pi}{2} (1-\nu) \quad\text{for } 1-\epsilon \le \nu \le 1.
\]

Since $|\eta\cdot x |\ge \frac{1}{2}$, we must have $\eta\cdot |x_i| \ge \frac{1}{2\sqrt{n}}$ for some $i$.
Since $|\eta\cdot x - \theta\cdot y| \le \frac{1}{4\sqrt{n}}$, we must have $\theta \cdot y_i \ge \frac{1}{4\sqrt{n}}$ if $x_i>0$ and $\theta\cdot y_i \le -\frac{1}{4\sqrt{n}}$ if $x_i<0$. 
Since  $1-\epsilon \le \eta,\theta \le 1$, Lemma \ref{der-lemma} shows that
\begin{align*}
	\left| \frac{\partial^kf}{\partial x_i^k}(\eta\cdot x) + \frac{\partial^kf}{\partial x_i^k} (\theta\cdot y)\right| 
 \ge \frac{1}{C} \cdot (1-\eta) + \frac{1}{C} \cdot (1-\theta)
 \gtrsim \frac{2}{\pi C} \left( \cos \frac{\pi\eta}{2} + \cos \frac{\pi\theta}{2} \right) 
 = \frac{2}{\pi C}(s+t). 
\end{align*}
Let $J$ be the the $k$-index with $j_i = k$ and $j_l =0$ for $l\ne i$. By Corollary \ref{cor-bb}, 
\begin{align*}
	d_{cc}(j_{\eta\cdot x} ^k(f) , j_{\theta\cdot y}^k(-f))  \gtrsim  
u_J(j_{\eta\cdot x} ^k(f) ^{-1} \odot j_{\theta\cdot y}^k(-f) ) 
 = \left| \frac{\partial^kf}{\partial x_i^k}(\eta\cdot x) + \frac{\partial^kf}{\partial x_i^k} (\theta\cdot y)\right|
\gtrsim \frac{2}{\pi C} (s+t).
\end{align*}
From (\ref{first-bound}), we may conclude
\[	
	d_{cc}(j_{\eta\cdot x} ^k(f) , j_{\theta\cdot y}^k(-f))  \gtrsim \rho_{\sS^n}   ((x\sin (\pi\eta/2),s), (y\sin (\pi\theta/2),-t)) .
\]
\phantom\qedhere
\end{proof}

%%%%%%%%%%%%%%%%%%%%%%%%%%%%%%%%%%%%%%%%%%%%%%%%%%%
%%%%%%%%%%%%%%%%%%%%%%%%%%%%%%%%%%%%%%%%%%%%%%%%%%%
%%%%%%%%%%%%%%%%%%%%%%%%%%%%%%%%%%%%%%%%%%%%%%%%%%%

\subsection{The embedding  does not admit a Lipschitz extension}\label{no-extension-section}

In this subsection, we will finish the proof of Theorem \ref{main-theorem}.
For the aid of the reader, we outline the remaining steps of the proof:\\

\noindent \textbf{Step 1:} 
Define the cylinder $C^{n+1} := \mathbb{B}^n\times [1,1]$ and construct a Lipschitz map $P:C^{n+1}\to \mathbb{B}^{n+1}$.

\noindent \textbf{Step 2:} 
We define the map $\lambda$ that shrinks $[-1,1]^n$ onto $\mathbb{B}^n$ by scaling line segments passing through the origin.
Show that $\lambda $ is invertible and Lipschitz. 
Then define $\Lambda:[-1,1]^{n+1}\to C^{n+1}$ by $\Lambda (x,t) = (\lambda (x),t)$. 

\noindent \textbf{Step 3:} 
Make sure that $f$ satisfies an integral condition, which may require slightly modifying $f$. 

\noindent \textbf{Step 4:} 
Suppose that $\phi$ admitted a Lipschitz extension $\tilde{\phi}$ and consider the Lipschitz constants of dilates of $\tilde{\phi}\circ P \circ \Lambda $ to arrive at a contradiction. \\

%%%%%%%%%%%%%%%%%%%%%%%%%%%%%%%%%%%%%%%%%%%%%%%%%%%
%%%%%%%%%%%%%%%%%%%%%%%%%%%%%%%%%%%%%%%%%%%%%%%%%%%
%%%%%%%%%%%%%%%%%%%%%%%%%%%%%%%%%%%%%%%%%%%%%%%%%%%

We first define a Lipschitz map that maps the cylinder $C^{n+1}:= \mathbb{B}^n\times[-1,1]$ onto $\mathbb{B}^{n+1}$. 
For some intuition, this map projects $\mathbb{S}^{n-1}\times [-1,1]$ onto $\mathbb{S}^{n-1}\times \{0\}$ and fixes $\{0\}^{n}\times [-1,1]$. 

\begin{definition}
Define $P:C^{n+1}\to \mathbb{B}^{n+1}$ by
\[
	P(\theta\cdot x,t) := (x\sin (\pi\theta/2), t\cos (\pi\theta/2)),
\]
where $x\in \mathbb{S}^{n-1}, \theta\in [0,1]$, and $-1\le t\le 1$.
\end{definition}

\begin{lemma}
 	The map $P:C^{n+1}\to \mathbb{B}^{n+1} $ is Lipschitz. 
\end{lemma}

\begin{proof}
	Via a rotation, it suffices to prove 
\begin{align*}
	\rho_{\B^{n+1}} ((\sin (\pi\eta/2) &,0 ,t\cos(\pi\eta/2)), (x\sin (\pi\theta/2), y \sin (\pi\theta/2),s\cos (\pi\theta/2)))\\
&\lesssim |(\eta,0,t) - (\theta x , \theta\cdot y , s)|,
\end{align*}
or equivalently
\begin{align*}
 \biggl| \sin (\pi\eta/2) &- x\sin (\pi\theta/2)\biggr| +|y\sin (\pi\theta/2)| + |t\cos (\pi\eta/2) - s\cos (\pi\theta/2)|\\
& \lesssim \rho_{\B^{n}}   ((\eta,0) , (\theta x , \theta\cdot y)) + |t-s|,
\end{align*}
where $-1\le x \le 1$, $(x,y) \in \sS^{n-1}$, $0\le \theta\le \eta$, $0<\eta$, and $-1\le s, t\le 1$. 

From the estimates performed in the proof of Proposition \ref{polar-coord},
\[
	 \biggl| \sin (\pi\eta/2) - x\sin (\pi\theta/2)\biggr| +\sin \left(\frac{\pi\theta}{2}\right) \sum_{i=2}^n |y_i| \lesssim \rho_{\B^{n}} ((\eta,0),  \theta\cdot (x,y))
\]
and
\begin{align*}
	|t\cos (\pi\eta/2) - s\cos (\pi\theta/2)| &\le |t\cos (\pi\eta/2) - s\cos (\pi\eta/2)| + |s\cos (\pi\eta/2) - s\cos (\pi\theta/2)|	\\
&\le |t-s| + |\cos (\pi\eta/2) - \cos (\pi\theta/2)|\\
&\lesssim |t-s| + \rho_{\B^{n}} ((\eta,0),  \theta\cdot (x,y)). 
\end{align*}
It follows that $\cL$ is Lipschitz. 
\end{proof}

%%%%%%%%%%%%%%%%%%%%%%%%%%%%%%%%%%%%%%%%%%%%%%%%%%%
%%%%%%%%%%%%%%%%%%%%%%%%%%%%%%%%%%%%%%%%%%%%%%%%%%%

We now consider the invertible map that shrinks $[-1,1]^{n+1}$ to $\mathbb{B}^{n+1}$ by scaling lines passing through the origin.

\begin{definition}\label{sector-lambda-def}
For $i=1,\ldots , n$, define
\[
	S_i := \{x\in [-1,1]^{n} : |x_i| \ge |x_j| \text{ for all } j\ne i\}.
\]	
Define $\lambda:[-1,1]^{n} \to \B^{n} $ by
\[
	\lambda(x) :=\left\{\begin{array}{cl}
 \frac{|x_i|}{|x|} \cdot x &\quad\text{if }  x \in S_i\setminus \{0\}\\
0 &\quad\text{if } x=0.
\end{array}\right.
\]
\end{definition}

Note that $[-1,1]^{n}$ is the  union of the $S_i$.
Also each $S_i$ is the disjoint union of two convex sets, the subset of $x$ with $x_i \ge 0$ and the subset with $x_i\le 0$. 

We now show that $\lambda$ is biLipschitz.

\begin{proposition}
The map $\lambda$ is invertible with $\lambda^{-1}:\B^{n}\to [-1,1]^{n}$ given  by
\[
	\lambda^{-1}(u) = \left\{ \begin{array}{cl}
 \frac{|u|}{|u_i|} \cdot u &\quad\text{if } u\ne 0 \text{ and } |u_i|\ge |u_j| \text{ for all } j\ne i\\
0 &\quad\text{if } u=0.
\end{array}\right.
\]
Moreover, $\lambda$ is biLipschitz  with
\[
	\frac{1}{3(n+1)} |x-y| \le |\lambda(x) - \lambda(y)| \le 3|x-y|, \quad x,y\in [-1,1]^{n+1}.
\]
\end{proposition}

\begin{proof}
We leave it to the reader to confirm that $\lambda$ is invertible with its inverse having the form as in the statement.

We  show that $\lambda$ is Lipschitz. 
Consider the case $x,y\in S_i$ for some common $i$. 
If $y=0$ or $x=0$, then 
\[
	|\lambda (x) - \lambda(y)| \le |x-y|.
\] 
If $x$, $y\ne 0$ are given with $|x|\le |y|$, 
\begin{align*}
	| \lambda(x) - \lambda(y) | 
& \le \left| \frac{|x_i|}{|x| }\cdot x  - \frac{|y_i|}{|y|} x \right| + \left| \frac{|y_i|}{|y|} x  - \frac{|y_i|}{|y|}\cdot y \right| \\
& \le \left| \frac{|x_i|}{|y|} \cdot ( |y| - |x| )  +  \frac{|x|}{|y|}\cdot |x_i| -   \frac{|x|}{|y|}\cdot |y_i| \right|  + |x-y| \\
& \le \frac{|x_i|}{|y|} \cdot |y-x| + \frac{|x|}{|y|} \cdot |x_i-y_i|  + |x-y| \\
&\le 3|x-y|. 
\end{align*}

For general $x,y\in [-1,1]^{n},$ let $\gamma:[0,1]\to [-1,1]^{n}$ be the straight line path connecting $x$ to $y$.
Fix a partition $0 = t_0< t_1< \ldots  < t_m = 1$ such that each restriction $\gamma|_{[t_j,t_{j+1}]}$ is contained in some $S_{i_j}$.
This is possible because each $S_i$ is the disjoint union of two convex sets.
Then
\[
	|\lambda(x) - \lambda(y)| \le \sum_{i=0}^{m-1} |\lambda(\gamma(t_{i+1}) ) -\lambda(\gamma(t_i))| \le \sum_{i=0}^{m-1} 3|\gamma(t_{i+1}) - \gamma(t_i)| = 3|x-y|.
\]
This proves that $\lambda$ is $3$-Lipschitz. 
The proof that $\lambda^{-1}$ is $\frac{1}{3n}$-Lipschitz is similar.  
\end{proof}

This enables us to define a map that stretches $C^{n+1}$ horizontally to $[-1,1]^{n+1}$ via $\lambda$.
Note that this map will be biLipschitz since $\lambda$ is.

\begin{definition}
	Define $\Lambda:[-1,1]^{n+1}\to C^{n+1}= \mathbb{B}^n\times [-1,1]$ by $\Lambda (x,t)  = (\lambda(x),t)$.
\end{definition}

%Note that $(P\circ \Lambda) (\partial [-1,1]^{n+1}) = \mathbb{S}^n$. 
%We make the following observation:
%\begin{proposition}
%	If $\phi:\mathbb{S}^n\to J^k(\R^n)$ admits a Lipschitz extension $\tilde{\phi}:\mathbb{B}%^{n+1}\to J^k(\R^n)$, then $\phi_1 := \phi\circ P\circ \Lambda|_{\partial [-1,1]^{n+1}}:\partial %[-1,1]^{n+1}\to J^k(\R^n)$ admits a Lipschitz extension $\widetilde{\phi_1} :[-1,1]^{n+1}\to J^k(\R^n)$.
%In fact, $\tilde{\phi}\circ L \circ \Lambda$ is one such extension. 
%\end{proposition}

We take a moment to note that why we choose to use $P\circ \Lambda$ to map a cube onto $\B^{n+1}.$
Note that $P\circ \Lambda$ maps the boundary of $[-1,1]^{n+1}$ onto the boundary of $\B^{n+1}$. 
This will set us up to replicate Rigot and Wenger's proof of Theorem 1.2 in \cite{RW:LNE} for the lack of a Lipschitz extension.
We could have used spherical coordinates to map a cube onto $\B^{n+1}$, but that would have been  more delicate since one would not have the ``mapping of boundaries".

The trickiest part of this proof will be ensuring that the smooth mapping $f:\R^n\to \R$ serving as the ``body'' of the embedding satisfies a nonzero integral condition. 
Before, we need to define integrals of Lipschitz forms on cubes and on the boundaries of cubes. 

\begin{definition}
	Let $g_1,\ldots , g_{n+1}: [-1,1]^{n+1} \to \R$ be Lipschitz functions.
We define 
\[
	\int_{[-1,1]^{n+1}} dg_1\wedge \cdots \wedge dg_{n+1} := \int_{[-1,1]^{n+1}} \det (\partial_{x_j}g_i) dx_1\cdots dx_{n+1}
\]
and 
\begin{align*}
	\int_{\partial [-1,1]^{n+1}} &g_1dg_2 \wedge \cdots \wedge dg_{n+1} \\
&:= \sum_{l=1}^{n+1} \int_{[-1,1]^n} \hat{g}_1^{l,1} \det (\partial_{x_j} \hat{g}_i^{l,1} )_{\substack{i\ge 2\\ j\ne l}} \ d\hat{x}_l - \int_{[0,1]^n} \hat{g}_1^{l,0} \det (\partial_{x_j} \hat{g}_i^{l,0} )_{\substack{i\ge 2 \\ j\ne l}}\ d\hat{x}_l,
\end{align*}
where $\hat{x}_l := (x_1,\ldots , x_{l-1},x_{l+1} ,\ldots , x_{n+1} )\in \R^n$ and 
$\hat{g}_i^{l,m} (\hat{x}_l) := g_i (x_1,\ldots , x_{l-1},m, x_{l+1},\ldots , x_{n+1}) $ for $m=-1,1$. 
\end{definition}

Rigot and Wenger's proof in \cite{RW:LNE} relies on a version of Stokes' Theorem for Lipschitz forms. 
\begin{lemma}\label{stokes-lemma-alternate}\cite[Lemma 3.3]{RW:LNE}
For all Lipschitz functions $g_1,\ldots , g_{n+1}: [-1,1]^{n+1} \to \R$,
\[
	\int_{[-1,1]^{n+1}} dg_1\wedge \cdots \wedge dg_{n+1} =\int_{\partial [-1,1]^{n+1}} g_1dg_2 \wedge \cdots \wedge dg_{n+1} .
\]
\end{lemma}

For the next proof, it will be helpful (to avoid repetition) if we set up notation for a function on $\partial [-1,1]^{n+1}$ obtained from a function on $\R^n$. 

\begin{notation}
For each smooth function $g:\R^n\to \R$, define $\bar{g}:\partial [-1,1]^{n+1}\to \R$ by
\[
	\bar{g}(x,t) = \left\{ \begin{array}{cl}
 g(\lambda(x)) &\quad\text{if } x\in [-1,1]^n \text{ and } t=1\\
-g(\lambda(x)) &\quad\text{if } x\in [-1,1]^n \text{ and } t = -1\\
g(\lambda(x)) &\quad\text{if } x \in \partial [-1,1]^n \text{ and } t\in (-1,1). 
\end{array}\right.
\]
\end{notation}

Note that if $g \equiv 0$ on $\mathbb{S}^{n-1}$, then $\bar{g}$ admits the Lipschitz extension $(x,t)\mapsto t g(\lambda(x))$ to $[-1,1]^{n+1}$. 

We now state the extra property we need  our function $f$ to satisfy. 

\begin{proposition}\label{f-with-integral-condition}
There exists a smooth function $f:\R^n\to \R$ satisfying:
\begin{enumerate}[label=(\alph*)]
\item $f(x) = (1-|x|)^{k+1} $ for $\frac{1}{2} \le |x|\le \frac{3}{2}$; and 
\item $f(x)>0$ for $|x|<1$.
\item $\int_{\partial [-1,1]^{n+1} }  \lambda_1d\lambda_2 \wedge \cdots \wedge d\lambda_n \wedge d\bar{f} \ne 0$,
where $\lambda_1,\ldots , \lambda_n$ are the components of $\lambda$.
\end{enumerate}
\end{proposition}

\begin{proof}
	By Lemma \ref{body-of-jet}, there exists a smooth function $f:\R^n\to \R$ satisfying properties (a) and (b). 
If $\int_{\partial [-1,1]^{n+1} }  \lambda_1d\lambda_2 \wedge \cdots \wedge d\lambda_n \wedge d\bar{f} \ne 0$, then $f$ works, so assume otherwise. 

Suppose $\beta$ is a smooth function supported in a cube inside $\{x \in S_1: x_1>0, \ |x| <\frac{1}{2}\}$ (recall $S_1 = \{x\in [-1,1]^n: |x_1 | \ge |x_j| \ \text{for all } j >1\}$). 
By linearity,
\begin{align*}
	\int_{\partial [-1,1]^{n+1} }  &\lambda_1d\lambda_2 \wedge \cdots \wedge d\lambda_n \wedge d\overline{(f+\beta)}  \\
&= \int_{\partial [-1,1]^{n+1} }  \lambda_1d\lambda_2 \wedge \cdots \wedge d\lambda_n \wedge d\overline{f} 
+ \int_{\partial [-1,1]^{n+1} }  \lambda_1d\lambda_2 \wedge \cdots \wedge d\lambda_n \wedge d\overline{\beta} \\
& =  \int_{\partial [-1,1]^{n+1} }  \lambda_1d\lambda_2 \wedge \cdots \wedge d\lambda_n \wedge d\overline{\beta} .
\end{align*}
Thus if we show the last integral is nonzero, then $f+\beta$ will work. 

As $\bar{\beta} \equiv 0 $ on $\partial [-1,1]^n \times [-1,1]$, 
 $\int_{\partial [-1,1]^{n}\times [-1,1] }  \lambda_1d\lambda_2 \wedge \cdots \wedge d\lambda_n \wedge d\overline{\beta}  =0$. 
We can simplify 
\begin{align*}
	\int_{\partial [-1,1]^{n+1} }  \lambda_1d\lambda_2 \wedge \cdots \wedge d\lambda_n \wedge d\overline{\beta} = 2 \int_{[-1,1]^n} \lambda_1 d\lambda_2 \wedge \cdots \wedge d\lambda_n \wedge d(\beta\circ \lambda). 
\end{align*}
Note that $\lambda^{-1} $ is smooth on $\text{int}(S_1)\cap \mathbb{B}^n$, where $\text{int}(S_1) $ is the interior of $S_1$. 
Hence,
\begin{align*}
	2 \int_{[-1,1]^n} \lambda_1 d\lambda_2 \wedge \cdots \wedge d\lambda_n \wedge d(\beta\circ \lambda) 
& = 2 \int_{\{x\in \text{int}(S_1): x_1>0\}} \lambda^*(u_1du_2 \wedge \cdots \wedge du_n \wedge d\beta) \ dx\\
& = 2(-1)^{n+1} \int_{ \{u\in \text{int}(S_1) \cap \mathbb{B}^n: u_1>0\}}  u_1\frac{\partial \beta}{\partial u_1}\cdot  J(\lambda^{-1})\ du,
\end{align*}
where we used that $\beta$ is supported in $\{u\in \text{int}(S_1) \cap \mathbb{B}^n: u_1>0\}$ for the first equality and change of coordinates for the second equality.
Integrating by parts,
\[
	\int_{ \{u\in \text{int}(S_1) \cap \mathbb{B}^n: u_1>0\}}  u_1\frac{\partial \beta}{\partial u_1}\cdot  J(\lambda^{-1})\ du = - \int_{ \{u\in \text{int}(S_1) \cap \mathbb{B}^n: u_1>0\}}  \frac{\partial (u_1\cdot J(\lambda^{-1})) }{\partial u_1}\cdot \beta \ du. 
\]
It remains to define $\beta$ carefully to ensure that the last integral is nonzero. 

For $ u \in \text{int}(S_1)$ with $u_1>0$ and $|u|<1$, one can calculate 
\[
	\frac{\partial \lambda^{-1}}{\partial u_1} (u)= -\frac{1}{u_1^2} \cdot |u|u + \frac{1}{u_1} \cdot \left( \frac{u_1}{|u|} \cdot u + |u|\cdot e_1\right)
\]
and
\[
	\frac{\partial \lambda^{-1}}{\partial u_i}(u) = \frac{1}{u_1} \cdot \left( \frac{u_i}{|u|} \cdot u + |u|\cdot e_i\right) , \quad i = 2,\ldots , n.
\]
In particular, $\frac{\partial \lambda^{-1}}{\partial u_j} (u_1,0,\ldots , 0) = e_j$ for $j=1,\ldots , n$ and $0<u_1<\frac{1}{2}$. 
Since $J(\lambda^{-1}) (u_1,0,\ldots , 0) = 1$ for $0<u_1<\frac{1}{2}$, 
\[
	 \frac{\partial (u_1\cdot J(\lambda^{-1})) }{\partial u_1}(1/4,0,\ldots ,0) = 1.
\]
By smoothness, there exists a cube $C\subset \{u\in S_1: |u| < \frac{1}{2}\}$ centered at $(1/4,0,\ldots ,0) $ on which $\frac{\partial (u_1\cdot J(\lambda^{-1})) }{\partial u_1}>0$. 
If $\beta:\R^n\to [0,1]$ is supported on $C$ and $\beta(1/4,0,\ldots , 0) = 1$, then 
\[
	 \int_{\partial [-1,1]^{n+1} }  \lambda_1d\lambda_2 \wedge \cdots \wedge d\lambda_n \wedge d\overline{\beta} = 2(-1)^n \int_{ \{u\in \text{int}(S_1) \cap \mathbb{B}^n: u_1>0\}}  \frac{\partial (u_1\cdot J(\lambda^{-1})) }{\partial u_1}\cdot \beta \ du \ne 0
\]
as desired and $f+\beta$ works. 
\end{proof}

%%%%%%%%%%%%%%%%%%%%%%%%%%%%%%%%%%%%%%%%%%%%%%%%%%%
%%%%%%%%%%%%%%%%%%%%%%%%%%%%%%%%%%%%%%%%%%%%%%%%%%%

Let $d_0$ be the Riemannian metric distance arising from defining an inner product on $Lie(J^k(\R^n))$ that makes the layer of the stratification orthogonal. 
Define $\iota:(J^k(\R^n),d_{cc}) \to (J^k(\R^n),d_0)$ to be the identity map, which is $1$-Lipschitz.
With the extra integral condition on $f$, we can prove that the corresponding embedding of $\mathbb{S}^n$ into $J^k(\R^n)$ does not admit a Lipschitz extension.

\begin{proof}[Proof of Theorem \ref{main-theorem}]
Fix a smooth function $f:\R^n\to \R$ satisfying properties (a)-(c) of Proposition \ref{f-with-integral-condition}, and let $\phi:\mathbb{S}^n\to J^k(\R^n)$ be the corresponding biLipschitz embedding (see Definition \ref{smooth-embedding-phi} and Theorem \ref{biLip-embedding}). 

Suppose, for contradiction, that $\phi$ admits a Lipschitz extension $\tilde{\phi}:\mathbb{B}^{n+1}\to J^k(\R^n)$.
Let $\lambda $ equal the Lipschitz constant $\text{Lip}(F)$ of the Lipschitz map $F:= \tilde{\phi} \circ P \circ \Lambda$. 
We show that for all $M>0$,
\begin{equation}\label{Lipschitz-inequalities}
	M^{1+\frac{k}{n+1}}\left|\int_{\partial [-1,1]^{n+1} }  \lambda_1d\lambda_2 \wedge \cdots \wedge d\lambda_n \wedge d\bar{f} \right|^{1/(n+1)}\le \text{Lip} (\iota \circ \delta_M \circ F) \le M\lambda. 
\end{equation}
Letting $M\to\infty$, we will arrive at a contradiction. 

The right inequality is clear since $\delta_M$ is $M$-Lipschitz and $\iota$ is $1$-Lipschitz. 

For the other inequality, let $h_i$ denote the $x_i$-coordinate of $F$ for $i=1,\ldots ,n$ and $h_{n+1}$ the $u_0$-coordinate of $\iota \circ \delta_M \circ F$.
%, where $(x,\tilde{u}^{(k)}, u_0)$ are global coordinates for $J^k(\R^n)$. 
For $(x,t) \in \partial [-1,1]^{n+1}$, $h_i(x,t) = M\lambda_i(x)$ for $i=1,\ldots , n$ and $h_{n+1}(x,t) = M^{k+1} \bar{f}(x)$.
This implies
\begin{equation}\label{stokes-equality-in-theorem}
	\int_{\partial [-1,1]^{n+1}} h_1dh_2 \wedge \cdots \wedge dh_{n+1} = 
M^{n+k+1} \int_{\partial[-1,1]^{n+1}} \lambda_1d\lambda_2 \wedge \cdots \wedge d\lambda_n \wedge d\bar{f} \ne 0.
\end{equation}
By Lemma \ref{stokes-lemma-alternate}, 
\[
\int_{\partial [-1,1]^{n+1}} h_1dh_2 \wedge \cdots \wedge dh_{n+1} = \int_{ [-1,1]^{n+1}} dh_1\wedge dh_2 \wedge \cdots \wedge dh_{n+1} .
\]

Define the $(n+1)$-form $\omega:= dx_1\wedge \cdots \wedge dx_n\wedge du_0$ on $J^k(\R^n)$.
By Lemma 3.2 of \cite{RW:LNE}, 
\[
	|\omega_p(v_1,\cdots ,v_{n+1})| \le 1
\]
for all $p\in J^k(\R^n)$ and $v_1,\ldots , v_{n+1}\in T_pJ^k(\R^n)$ with $||v_i||_{g_0} \le 1$.
We have
\[
\left|\int_{ [-1,1]^{n+1}} dh_1\wedge dh_2 \wedge \cdots \wedge dh_{n+1}  \right|= \left| \int_{[-1,1]^{n+1} }(\iota\circ \delta_M \circ F)^*\omega\right| \le Lip(\iota \circ \delta_M \circ F)^{n+1}.
\]
The left inequality of (\ref{Lipschitz-inequalities}) follows from (\ref{stokes-equality-in-theorem}).
We may conclude that $\phi$ does not admit a Lipschitz extension to $\mathbb{B}^{n+1}$. 
\end{proof}

%%%%%%%%%%%

%$(x,t) \in \partial [-1,1]^{n+1}$

%$\Lambda(x,t)  = (\lambda(x),t)$

%Say $\lambda (x) = \theta\cdot y $ (either $\theta = 1$ or $t \in \{-1,1\}$)

%$P (\Lambda(x,t)) = (y \sin(\pi\theta/2), t \cos (\pi\theta/2))\in \mathbb{S}^n$

%$\phi(P(\Lambda(x,t))) =  j_{\theta\cdot y}^k(\pm f) = j_{\lambda(x)}^k(\pm f)$

%%%%%%%%%%%%%%%%%%%%%%%%%%%%%%%%%%%%%%%%%%%%%%%%%%%
%%%%%%%%%%%%%%%%%%%%%%%%%%%%%%%%%%%%%%%%%%%%%%%%%%%
%%%%%%%%%%%%%%%%%%%%%%%%%%%%%%%%%%%%%%%%%%%%%%%%%%%

\bibliographystyle{plain}

\bibliography{biLip_bib}

\end{document}